\documentclass[reqno]{amsart}
\usepackage{amsmath,amsthm}
\usepackage{amssymb}
 \usepackage[all,cmtip]{xy}

\newtheorem{Thm}[subsection]{Theorem}
\newtheorem{Lem}[subsection]{Lemma}
\newtheorem{Prop}[subsection]{Proposition}
\newtheorem{Cor}[subsection]{Corollary}

\newtheorem{Exm}[subsection]{Example}
\newtheorem{Rem}[subsection]{Remark}

\newcommand{\bcp}{\mathbb C \mathbb P}

\newcommand{\wt}{\widetilde}

\begin{document}

\title[Almost Complex structure on $S^{2m} \times M$]{Nonexistence of  Almost Complex Structures on the product $S^{2m} \times M$}

\author{Prateep Chakraborty}
\address{Stat-Math Unit, Indian Statistical Institute, 8th Mile, Mysore Road, RVCE Post, Bangalore 560059, INDIA.}
\email{chakraborty.prateep@gmail.com}

%
%\author{Aniruddha C. Naolekar}
%\address{Stat-Math Unit, Indian Statistical Institute, 8th Mile, Mysore Road, RVCE Post, Bangalore 560059, INDIA.}
%\email{ani@isibang.ac.in}

\author{ Ajay Singh Thakur}
\address{Stat-Math Unit, Indian Statistical Institute, 8th Mile, Mysore Road, RVCE Post, Bangalore 560059, INDIA.}

\email{thakur@isibang.ac.in}

\thanks{The research of second author is supported by DST-Inspire Faculty Scheme(IFA-13-MA-26)}


\begin{abstract}
In this note we give a necessary  condition for having an almost complex structure on the product $S^{2m} \times M$, where $M$ is a connected orientable closed manifold. We show that if the Euler characteristic $\chi(M) \neq 0$, then except for finitely many values of $m$, we do not have almost complex structure on $S^{2m} \times M$. In the particular case when $M = \bcp^n, n \neq 1$, we show  that if $n \not \equiv 3 \pmod 4$ then $S^{2m} \times \bcp^{n}$ has an almost complex structure if and only if $m = 1,3$. As an application we obtain  conditions on the nonexistence of  almost complex structure on Dold manifolds. 
\end{abstract}
\maketitle

\section{Introduction}
Recall that an oriented manifold $X$ has an almost complex structure (a.c.s. for short) if there is a complex vector bundle $\xi$ such that the underlying real bundle $\xi_{\mathbb R}$ of $\xi$ is isomorphic to  the tangent bundle $\tau_X$ of $X$ as oriented bundles.
% A natural question for an even dimensional oriented manifold $X$ is the existence of an a.c.s. on $X$.

%  We shall give some necessary conditions on having an a.c.s. on $S^{2m} \times M$ in terms of the Euer characteristic $\chi(S^{2m} \times M)$ and hence the condition will be independent of the orientation we choose on $S^{2m} \times M$.
It is well known \cite{borel} that the only even dimensional spheres admitting an a.c.s are $S^{2}$ and $S^6$. The product of even dimensional spheres, $S^{2m} \times S^{2n}$ with $m,n \neq 0$, has an a.c.s. if and only if  $(m,n) = (1,1),(1,3),(3,1),(1,2),(2,1) \mbox{ or } (3,3)$ (see, for example, \cite{datta} and \cite{sutherland}).   In \cite{tang}, Tang showed that $S^{2m} \times \bcp^2$ has an  a.c.s. if and only if $m = 1,3$, and $S^{2m} \times \bcp^3$ has an a.c.s. if and only if $m =  1,2,3$.

 In this note we deal with the question of existence of an a.c.s. on the general product of the form $S^{2m} \times M$, where $M$ is a connected orientable closed manifold of even dimension and $m \neq 0$. We have the following main result.
\begin{Thm}\label{4m}
	If $S^{2m} \times M$ has an a.c.s. then $2^r\cdot(m-1)!$ divides the Euler characteristic $\chi(S^{2m} \times M)$, where $2^r$ is the highest power of $2$ dividing $m$. 
\end{Thm}

 In the special case when $M = \bcp^n$, we  have the following result.
\begin{Thm}\label{finalthm}
	Let $n>1$ and $n \not \equiv 3 \pmod 4$. Then $S^{2m} \times \bcp^{n}$ has a.c.s. if and only if $m = 1,3$.		
\end{Thm}
We also observe that $S^2 \times \bcp^1$ has exactly two almost complex structures whereas $S^2 \times \bcp^2$ and  $S^4 \times \bcp^3$ have infinitely many almost complex structures. 
%An a.c.s. on $S^{2m+1} \times N$ where $N$ is an odd dimensional  Lens space has been studied in \cite{thomascount} and \cite{kob}.
 As an application we obtain a  result  on the nonexistence of a.c.s. on Dold manifolds which strengthens the result obtained  in \cite{tangdold}.

For any real or complex vector bundle $\xi$ over a $CW$-complex $X$,  we shall denote its stable class in $\wt{KO}(X)$ or $\wt{K}(X)$ by $[\xi]$. Let $Vect_n(X)$ be the set of all isomorphism classes of $n$-rank complex vector bundles over $X$. One of the  stability properties of vector bundles over $X$ is that if $X$ is of dimension $2n$ then  the map $Vect_n(X) \rightarrow \wt{K}(X)$ which takes the complex vector bundle $\xi$ to its stable class $[\xi]$ is bijective (see \cite[Chapter 9, Theorem 1.5]{hus}).
If $\rho: \wt{K}(X) \rightarrow \wt{KO}(X)$ is the realization map then the proof of our result is based on the following theorem.

\begin{Thm} {\em (\cite[Theorem 1.1]{sutherland} and \cite[Lemma 1,1]{thomascount})} \label{euler}
A $2n$-dimensional connected oriented manifold $X$ admits an a.c.s. if and only if there exists an element $a \in \wt{K}(X)$ such that $\rho(a) = [\tau_X]$ and $nth$-Chern class $c_n(a)$ equals to the Euler class $e(X)$. \qed
\end{Thm}

In this note, if $X$ is a complex manifold then the holomorphic tangent bundle of $X$ will be denoted by $T_X$. All the manifolds $M$ we consider here will be connected, closed, orientable and of dimension $2n$.

\section{Almost complex structures on $S^{2m} \times M$}
%In this section we shall fix our notations and  recall some basic results in $K$-theory which will be used to prove our results.
% 
% We have the following well-known fact:
%
%An oriented smooth manifold $X$
% of dimensional $2n$ admits an a.c.s. if and only if  there is a complex vector bundle $\xi$ over $X$ whose the $n$-th Chern class $c_n(\xi)$ equals to the Euler class of $X$.
% 
We first fix our notations. Let $\nu$ be the canonical complex line bundle over $S^2$. Let $g = (\nu -1) \in \wt{K}(S^2)$ be a generator. Then by the Bott-Periodicity, $g^m \in \wt{K}(S^{2m})$ is  a generator. We fix the generator $y_m \in H^{2m}(S^{2m};\mathbb Z)$ such that the total Chern class $c(g^{m}) = 1+ (m-1)!\cdot y_m$. The Bott Integrability Theorem \cite[Chapter 20, Corollary 9.8]{hus} says that for any $a \in \wt{K}(S^{2m})$, the top Chern class $c_m(a)$ is divisible by $(m-1)!$.  (Here divisibility is in the sense that $c_m(a) = (m-1)!\cdot e_i$ for some $e_i \in H^{2m}(S^{2m};\mathbb Z)$). 

In the following proposition we shall give a  condition on the Chern classes of any vector bundle over the smash product $S^{2m} \wedge M$.

\begin{Prop}\label{suspension}
For an element $a \in \wt{K}(S^{2m} \wedge M)$, each Chern class $c_i(a)$ is divisible by $(m-1)!$.	
\end{Prop}
\begin{proof}
Let $a \in \wt{K}(S^{2m} \wedge M)$. Then by the Bott-Periodicity, $a = g^{m} \otimes (\beta -n)$ for some complex vector bundle $\beta$ of rank $n$ over $M$. Let $c(\beta) = \prod_k (1 + s_k)$ be the formal factorization obtained by using the splitting principle. As in the proof of Lemma 2.1(2) \cite{tanaka}, we can formally write the total Chern class of $a$ as follow
\begin{equation} \label{beta}
\begin{array}{rl}  c(a) = & c(g^{m}\otimes(\beta-n)) \\ \\
 = & \prod_k(1+ (m-1)! \cdot y_m((1+s_k)^{-m}-1))\\ \\
 = & 1+(m-1)! \cdot y_m(\sum_{i\geq1}(-1)^i \binom{m+i-1}{i}(\sum_{k}s_k^i)).
\end{array}
\end{equation}
The last equality in Equation \ref{beta} is due to the fact that $y_m^2 = 0$. The summation $\sum_k s_k^i$ can be expressed as a polynomial  in Chern classes of $\beta$  with integer coefficients (by Newton-Girard formula). From  here it is clear that each Chern class $c_i(a)$ is divisible by $(m-1)!$. \end{proof}

Next we shall obtain conditions on Chern classes of any element in $\wt{K}(S^{2m} \times M)$. Consider the following split exact sequence
$$0\rightarrow \wt{K}(S^{2m} \wedge M) \rightarrow \wt{K}(S^{2m} \times M) \rightarrow \wt{K}(S^{2m}) \oplus  \wt{K}(M) \rightarrow 0.$$
 Hence any  element $a \in \wt{K}(S^{2m} \times M)$  can be written as $a = a_1 + a_2 + a_3,$
where $a_1 \in \wt{K}(S^{2m} \wedge M)$, $a_2 \in \wt{K}(S^{2m})$ and $a_3 \in \wt{K}(M)$. In the following proposition we choose an orientation on $S^{2m} \times M$.

\begin{Prop}\label{1st prop}
For any $a \in \wt{K}(S^{2m} \times M)$,  all the Chern numbers of $a$ are divisible by $(m-1)!$. 
\end{Prop}
\begin{proof}
Let $ a = a_1 + a_2 + a_3$ be as described above. Let the total Chern class $c(a_1) = 1 + y_m x_1 + y_m x_2 + \cdots + y_m x_n$, $c(a_2) = 1+by_m$ and $c(a_3) = 1+z_1+z_2+\cdots+z_n$ where $x_i, z_i \in H^{2i}(M;\mathbb Z)$ and $b$ is an integer. Since $a_1 \in \wt{K}(S^{2m} \wedge M)$, by Proposition \ref{suspension}, it is clear that each $y_m x_i$ is divisible by $(m-1)!$. Also since $a_2 \in \wt{K}(S^{2m})$, by Bott Integrability Theorem, we have that $b$ is divisible by $(m-1)!$. Now the proof of the proposition follows easily.
\end{proof}

%The following result now follows immediately which gives a necessary condition on the Euler characteristic of $M$ to have an a.c.s. on  $S^{2m} \times M$.
%\begin{Thm} \label{1st Thm}
%If  $S^{2m} \times M$  has an a.c.s. then $(m-1)!$ divides the Euler characteristic $\chi(S^{2m}\times M) = 2\chi(M)$. \qed
%\end{Thm}

%From the above theorem it is clear that if $m \geq 5$ and $\chi(M)$ is odd then there is no a.c.s. on $S^{2m} \times M$.

To prove our main result consider the following  commutative diagram
$$\xymatrix@C = .1 cm{
\wt{K}(S^{2m} \times M ) \ar[d]^{\rho} & = & \wt{K}(S^{2m} \wedge M) \ar[d]^{\rho}  & \oplus & \wt{K}(S^{2m}) \ar[d]^{\rho} & \oplus & \wt{K}(M) \ar[d]^{\rho}\\
\wt{KO}(S^{2m} \times M ) & = &\wt{KO}(S^{2m} \wedge M) & \oplus & \wt{KO}(S^{2m}) & \oplus & \wt{KO}(M) 
}$$
In view of Theorem \ref{euler}, if $a \in \wt{K}(S^{2m} \times M)$ gives  an a.c.s. on $S^{2m} \times M$ then $\rho(a) = [\tau_{S^{2m} \times M}]$  and the top Chern class  $c_{m+n}(a) = e(S^{2m} \times M)$.  If we write  $a = a_1 + a_2 +a_3$ then using the fact that $[\tau_{S^{2m} \times M}] = [\tau_M]$, we have $\rho(a_1) = 0$, $\rho(a_2) = 0$, $\rho(a_3) = [\tau_{M}]$. 

%Let $a_1\in\wt{K}(S^{2m}\wedge M)$, then by Bott-Periodicity $a_1=g^{m}\otimes\beta,$ where $\beta\in\wt{K}(M).$ From Splitting principle, we shall get a space $M_1$ satisfying the following commutative diagram\\
%$$
%\begin{array}{ccc}
 %\beta_1+\cdots+\beta_n & \rightarrow & \beta\\
 %\downarrow & & \downarrow\\
 %M_1 &\stackrel{\phi}\longrightarrow & M
%\end{array}$$
%where each $\beta_i$ is a line bundle over $M_1$ and $\phi^*:H^*(M)\to H^*(M_1)$ is injective.\\
%Thus, $\phi^*(c_i(\beta))=\sum_{k_1<\cdots<k_i}c_1(\beta_{k_1})\cdots c_1(\beta_{k_i}).$ Now, the Chern class of $g^{m}\otimes(\beta_1+\cdots+\beta_n)$ is 
%$$\begin{array}{cll}
% & c(g^{m}\otimes(\beta_1+\cdots+\beta_n)) \\
 %= & \prod_k(1+c_{m}(g^{m})(1+(1+c_1(\beta_k))^{-m}-1))\\
% = & 1+c_{m}(g^{m})\sum_{i\geq1}(-1)^i{m+i-1\choose i}(\sum_{k}c_1^i(\beta_k))\\
 %= & (id\wedge\phi)^*(1+(m-1)!y\sum_{i\geq1}(-1)^i{m+i-1\choose i}P_i(c_1(\beta),\cdots,c_i(\beta))),
%\end{array}$$
%where $P_i$'s are polynomials in $c_1(\beta),\cdots,c_i(\beta).$\\
%And $\phi^*$ is injective implies
%\begin{eqnarray}\label{beta1}
%c(g^{2m}\otimes\beta)=1+(m-1)!y\sum_{i\geq1}(-1)^i{m+i-1\choose i}P_i(c_1(\beta),\cdots,c_i(\beta)).
%\end{eqnarray}

%
%\begin{Prop}
%If $S^{4m} \times M$ has an a.c.s. then $(2m-1)!$  divides $\chi(M)$.
%\end{Prop}

\begin{proof}[{\bf Proof of Theorem \ref{4m}}]
Suppose $a \in \wt{K}(S^{2m} \times M)$ gives an a.c.s. on $S^{2m} \times M$. From  Proposition \ref{1st prop} and as $c_{m+n}(a)  = e(S^{2m} \times M)$, it follows that $(m-1)!$ divides $\chi(S^{2m} \times M)$. Hence in the case when $m$ is odd, the proof of the theorem is complete. 
 
Next assume that $m = 2p$.  As above we write $a= a_1 + a_2 + a_3$.  First observe that  the realization map $\rho: \wt{K}(S^{4p}) \rightarrow \wt{KO}(S^{4p})$ is injective, and since $\rho(a_2) = 0$, we have $a_2 =0$. 
Next note that  $\rho(a_1) = 0$ and this implies that $a_1 + \overline{a_1} =0$. Hence $2c_{2i}(a_1) = 0$ for $i > 0$ as the non-trivial cup-products in the cohomology of the smash product $S^{4p} \wedge M$ are zero. Therefore the top Chern class 
\begin{equation}\label{eq2}
c_{2p +n}(a) = \sum_{2i+1+j= 2p+n}c_{2i+1}(a_1)c_j(a_3).
\end{equation} 
%From the Remark \ref{remark1}, we know that $2\cdot(2m-1)!$ divides each odd Chern class $c_{2i+1}(a_1)$.  
%Finally, as in the proof of Proposition \ref{1st prop}, we now conclude that $(2m-1)!$ divides  $\chi(M)$. 
%Therefore, using the fact that ${r\choose s}$ is even, when $r$ is even and $s$ is odd, we get from \ref{beta} that $2(2m-1)!$ divides each Chern class of $a_1.$ Again, $\rho(a_2)=0~\Rightarrow a_2=0$. So,
%$$c(\xi)=(1 + e_1x_1y + e_2x_2y + \cdots + e_nx_ny)(1+z_1+z_2+\cdots+z_n),$$
%which shows that $2(2m-1)!$ divides the Euler characteristic $2\chi(M)$, so $(2m-1)!$ divides $\chi(M).$
Now consider the element $a_1^\prime \in \wt{K}(S^{4p+2}\wedge M)$ given as $a_1^\prime=g\otimes a_1$.  We can write $a_1 = g^{2p}\otimes(\beta -n)$ for some complex vector bundle $\beta$ of rank $n$ over $M$ and let   $c(a_1) = 1 + y_{2p}x_1 + y_{2p}x_2 + \cdots + y_{2p} x_n$. Then as in the proof of Proposition \ref{suspension}, the total Chern class of $a_1^\prime$ will be 
$$c(a_1^{\prime}) =  c(g \otimes a_1) =  1+y_{2p+1}(\sum_{i\geq1}(2p+i)x_i).$$
%$$\begin{array}{ccl}
% c(a_1^\prime) & = & c(g^{m+1}\otimes(\beta-n)) \\
%
% &= & 1+m! ~y_{m+1}(\sum_{i\geq1}(-1)^i{m+i\choose i}(\sum_{k}s_k^i))\\
% &= & 1+y_{m+1}(\sum_{i\geq1}(m+i)x_i)
% \end{array}$$
 Since $a_1^{\prime} \in \wt{K}(S^{4p+2} \wedge M)$, by Proposition \ref{suspension}, we have that each $(2p+i)y_{2p+1}x_i$ is divisible by $(2p)!$. This immediately implies that $(2p)!$ divides $(2p+i)y_{2p}x_i$.  Hence for each odd $i$, $(2p)!$ will divide  $$(\prod_j(2p+j)) \cdot y_{2p} x_i,$$ where the product varies over odd $j$ and  $1 \leq j \leq n$. Hence by Equation \ref{eq2} we have that $(2p)!$ divides $$(\prod_j (2p+j))\cdot\chi(S^{4p} \times M),$$ where  $j$ is odd and $1 \leq j \leq n$.
% Thus, if $2^r$ is the highest power of $2$ dividing %$2m$ then $(2m-1)!~2^r$ divides  $\chi(S^{4m}\times M).$ 
Now  the proof of the theorem follows from the fact that $(2p-1)!$ divides $\chi(S^{4p} \times M)$.
\end{proof}

%For an integer $n$, let  $f(n) = \mbox{max}\{r~ | ~2^r \mbox{ divides } n!\}$. Then we have the easy consequence of the above discussion.
%
%\begin{Prop}
%Let $\chi(M) = 2^kp$ and $p$ odd. If $S^{2m} \times M$ has an a.c.s. then $f(m-1) \leq k+1$ and further, if $m$ is even then $f(m) \leq k+1$. \qed
%\end{Prop}
%
%We have the following corollaries to the above proposition.
%
%
%
%
%
%\begin{Cor}\label{cor}
%Let  $\chi(M) = 2^rt +2^{r-1}$ for $t \neq 0$. If $S^{2m} \times M$ has an a.c.s. then $f(m-1) \leq r+1$ and further, if $m$ is even then  $f(m) \leq r+1$. \qed
%\end{Cor} 

From  Theorem \ref{4m} we can observe  that for a  given $M$ with $\chi(M) \neq 0$, except for finitely many values of $m$, we do not have a.c.s on $S^{2m} \times M$. In particular we have the following corollary whose proof follows immediately.

\begin{Cor} \label{cor2}
Let  $\chi(M) \not \equiv 0 \pmod 4$ or $\chi(M)$ be  a power of two. Then $S^{2m} \times M$ does not have an a.c.s.  for  $m \neq 1,2,3$. \qed
\end{Cor}

The restriction on $m$ in Corollary \ref{cor2} is the best possible as we know that $S^2 \times \bcp^n$, $S^6 \times \bcp^n$, $S^4 \times S^2$ and $S^4 \times \bcp^3$ (see Example \ref{example3} below) admits an a.c.s. 

%egin{proof}Let $S^{2m}\times M$ have an a.c.s..
% Then by Theorem \ref{1st theorem} we have that $(m-1)!$ divides $2\chi(M)$. Since,  $\chi(M)\not\equiv 0\pmod 4$, this implies that $m < 5$. From the above discussion, we also observe when  $m=4,$ we have that  $2\chi(M)$ is divisible by 8, which is not possible as $\chi(M)\not\equiv 0\pmod 4$. This completes the proof of the proposition.
%\end{proof}

\section{Almost complex structure on $S^{2m} \times \bcp^n$}

In this section we shall deal with the case when $M$ is a complex projective space $\bcp^n$. By Corollary \ref{cor2} we know that in the case when $n \not \equiv 3\pmod 4$ or $n+1 = 2^k$ for some $k\geq 0$, the product $S^{2m} \times \bcp^n$ does not admit a.c.s. for $m \neq 1,2$ and $3$.  We shall prove that if $n \not \equiv 3\pmod 4$ and $n >1$ then $S^4 \times \bcp^n$ also does not admit a.c.s.

 Let $H$ be the canonical line bundle over $\bcp^n$. Let $x = c_1(H) \in H^2(\bcp^n;\mathbb Z)$ be a generator.  Let $\eta = H -1 \in \wt{K}(\bcp^n)$ and  $r = [n/2]$. For $\alpha \in \wt{K}(S^{2m})$ and $\beta \in \wt{K}(\bcp^n)$ we shall write the external product $\alpha \otimes \beta \in \wt{K}(S^{2m} \wedge \bcp^n)$ as $\alpha \beta$. In view of Lemma 3.5  \cite{fujii1} we have the following lemma whose proof follows from the Bott-Periodicity.

\begin{Lem}
Each of the following system of elements form an integral basis of $\wt{K}(S^{2m} \wedge \bcp^n)$.

(i) $g^m\eta$, $g^m\eta(\eta + \bar{\eta})$, $\cdots$, $ g^m\eta(\eta + \bar{\eta})^{r-1}$, $g^m(\eta + \bar{\eta})$, $g^m(\eta + \bar{\eta})^2$, $\cdots$, $g^m(\eta + \bar{\eta})^r$, and also in case $n$ is odd, $g^m\eta^{2r+1} = g^m\eta(\eta+\bar{\eta})^r$;
 
 \vspace{ .5cm}
 
(ii) $g^m\eta$, $g^m\eta(\eta + \bar{\eta})$, $\cdots$, $ g^m\eta(\eta + \bar{\eta})^{r-1}$, $g^m(\eta -\bar{\eta})$, $g^m(\eta -\bar{\eta})(\eta + \bar{\eta})$, $\cdots$, $g^m(\eta -\bar{\eta})(\eta + \bar{\eta})^{r-1}$, and also in case $n$ is odd, $g^m\eta^{2r+1} = g^m\eta(\eta+\bar{\eta})^r$. \qed
\end{Lem}

 Let $w_k = g^m(H^k -1) - \overline{g^m(H^k-1)} \in \wt{K}(S^{2m} \wedge \bcp^n) $. Similar to Proposition 4.3 \cite{thomas}, we have the following three propositions.

\begin{Prop}\label{prop1}
Let $m \equiv 1,3 \pmod 4$. Then the kernel of the realization map $\rho: \wt{K}(S^{2m} \wedge \bcp^n) \rightarrow \wt{KO}(S^{2m} \wedge \bcp^n)$ is freely generated by
$w_1,w_2,\cdots,w_r$.
\end{Prop}

\begin{proof}
It is clear that all $w_k$'s are in the kernel of the realization map. Further it can be proved using the the fact that $\bar{g} = -g$ in $\wt{K}(S^2)$ and using  induction on $k$  that

$$ g^m(\eta + \bar{\eta})^{k} = w_{k} + \mbox{(linear combination of $w_1, \cdots, w_{k-1}$)}.$$ Now using the structure of $\wt{KO}^{-2m}(\bcp^n)$ (see \cite{fujii}) and the fact that the $\mathbb Q$-linear map $$\rho \otimes Id: \wt{K}(S^{2m} \wedge \bcp^n)\otimes_{\mathbb Z} \mathbb Q \rightarrow \wt{KO}(S^{2m} \wedge \bcp^n)\otimes_{\mathbb Z} \mathbb Q$$ is surjective, we can see that $w_1,\cdots, w_{k}$ freely generates the kernel of $\rho$.
\end{proof}

For next two propositions we first note that if $m$ is even then $$ g^m(\eta - \bar{\eta})(\eta + \bar{\eta})^{k-1} = w_{k} + \mbox{(linear combination of $w_1, \cdots, w_{k-1}$)},$$ whose proof can again be given using induction on $k$. Next consider the following commutation diagram,
$$\xymatrix@C = .3cm{
\wt{K}(S^{2m} \wedge S^{2n})  \ar[d]^{\rho} \ar[r] & \wt{K}(S^{2m} \wedge \bcp^n) \ar[d]^{\rho} \ar[r] & \wt{K}(S^{2m} \wedge \bcp^{n-1})  \ar[d]^{\rho} \\
\wt{KO}(S^{2m} \wedge S^{2n} )  \ar[r] & \wt{KO}(S^{2m} \wedge \bcp^n) \ar[r] & \wt{KO}(S^{2m} \wedge \bcp^{n-1}) 
}$$
Now the proof of the following two propositions follow from the fact that the realization map $ \rho: \wt{K}(S^{2l}) \rightarrow \wt{KO}(S^{2l})$ is nonzero for $l \not \equiv 3 \pmod 4$.
\begin{Prop}\label{prop2}
Let $m \equiv 0 \pmod 4$. Then the kernel of the realization map $\rho: \wt{K}(S^{2m} \wedge \bcp^n) \rightarrow \wt{KO}(S^{2m} \wedge \bcp^n)$ is freely generated by 
\begin{enumerate}
\item $w_1,w_2,\cdots,w_r$ when $n$ is even;

\item $w_1,w_2,\cdots, w_r, g^m \eta^n$, when $n \equiv 3 \pmod 4$;

\item $w_1,w_2,\cdots w_r,2g^m\eta^n$, when $n \equiv 1 \pmod 4$. \qed
\end{enumerate} 
\end{Prop}

\begin{Prop}
Let $m \equiv 2 \pmod 4$. Then the kernel of the realization map $\rho: \wt{K}(S^{2m} \wedge \bcp^n) \rightarrow \wt{KO}(S^{2m} \wedge \bcp^n)$ is freely generated by 
\begin{enumerate}
\item $w_1,w_2,\cdots,w_r$ when $n$ is even;

\item $w_1,w_2,\cdots, w_r, g^m \eta^n$, when $n \equiv 1 \pmod 4$;

\item $w_1,w_2,\cdots w_r,2g^m\eta^n$, when $n \equiv 3 \pmod 4$.  \qed
\end{enumerate}
\end{Prop}

Next we shall describe the Chern classes of elements in the Kernel of the realization map $\rho: \wt{K}(S^{2m} \wedge \bcp^n) \rightarrow \wt{KO}(S^{2m} \wedge \bcp^n)$. Using Lemma 2.1(2) \cite{tanaka} and by the fact that $y_m^2 =0$, we  can easily compute the total Chern class of $w_k$. When $m$ is even then 
$$c(w_k) = 1 -(m-1)!\cdot\sum_{i\geq 1}2\binom{m+2i-2}{2i-1} k^{2i-1}y_m x^{2i-1}.$$
When $m$ is odd we have 
$$c(w_k) = 1 +(m-1)!\cdot\sum_{i\geq 1}2\binom{m+2i-1}{2i} k^{2i}y_m x^{2i}.$$
To compute the Chern class of $g^m\eta^n \in \wt{K}(S^{2m} \wedge\bcp^n)$, consider the following  exact sequence 
$$0 \rightarrow \wt{K}(S^{2m} \wedge S^{2n}) \rightarrow \wt{K}(S^{2m} \wedge\bcp^n) \rightarrow \wt{K}(S^{2m} \wedge\bcp^{n-1}).$$ The element $g^m\eta^n \in \wt{K}(S^{2m} \wedge\bcp^n)$ is a generator of the kernel of the last map in the above sequence. Hence the total Chern class of $g^m\eta^n$ is as follows:
$$\begin{array}{lcl}
c(g^m\eta^n) & = & 1 \pm (m+n-1)! \cdot y_{m}x^n\\ \\
             & = & 1 \pm (m-1)!\cdot n! \cdot \binom{m+n-1}{n}y_mx^n. 
\end{array}$$

We have the following proposition whose proof follows from the above discussion.

\begin{Prop}\label{chernclasses}
The total Chern class of any element in the kernel of the realization map $\rho: \wt{K}(S^{2m} \wedge \bcp^n) \rightarrow \wt{KO}(S^{2m} \wedge \bcp^n)$
is of the form

\vspace{.5 cm}

\begin{enumerate}
\item $1 + 2\cdot(m-1)! \cdot \big [ \sum_{i\geq 1} \binom{m+ 2i -1}{2i} (\sum_{k =1}^{r} b_k k^{2i})y_mx^{2i} \big ]$, when $m$ is odd. 

\vspace{.5 cm}

\item $1 - 2\cdot(m-1)! \cdot \big[ \sum_{i\geq 1} \binom{m+ 2i -2}{2i-1} (\sum_{k =1}^{r} b_k k^{2i-1})y_mx^{2i-1} \big ]$, when $m,n$ are even. 

\vspace{.5 cm}

\item $1 - (m-1)! \cdot \big [2 \sum_{i = 1}^{r+1}\binom{m+2i-2}{2i-1} (\sum_{k=1}^{r} b_k k^{2i-1})y_mx^{2i-1} \pm \binom{m+n-1}{n} \cdot n!\cdot b_{r+1}y_mx^n \big ],$ when $m \equiv 0 \pmod 4$ and $n \equiv 3  \pmod 4$, or $m \equiv 2 \pmod 4$ and $n \equiv 1  \pmod 4$. 

\vspace{.5 cm}

\item $1 - 2\cdot(m-1)!\cdot \big [\sum_{i=1}^{r+1}\binom{m+2i-2}{2i-1} (\sum_{k=1}^{r} b_k k^{2i-1})y_mx^{2i-1} \pm \binom{m+n-1}{n} \cdot n! \cdot b_{r+1}y_mx^n \big ],$ when $m \equiv 0 \pmod 4$ and $n \equiv 1  \pmod 4$, or $m \equiv 2 \pmod 4$ and $n \equiv 3 \pmod 4$,

\vspace{.5 cm}

for  $b_1, b_2, \cdots, b_r, b_{r+1} \in \mathbb Z$. \qed
\end{enumerate}
\end{Prop}

\begin{Rem} \label{remark1}{\em
In the above proposition, using that the fact that a binomial coefficient  $\binom{s}{t}$ is even if $s$ is even and $t$ is odd,
 one can see that if $m$ is even and $n>1$ then for any $a_1 \in \wt{K}(S^{2m} \wedge \bcp^n)$ such that $\rho(a_1) = 0$ we have that $4\cdot(m-1)!$ divides each Chern class $c_i(a_1)$. }\end{Rem}

\begin{Prop}\label{2nd prop}
If $S^{4p} \times \bcp^n$ has a.c.s. then $2\cdot(2p-1)!$ divides $\chi(\bcp^n) = (n+1)$.
\end{Prop}
\begin{proof}
Suppose $a \in \wt{K}(S^{4p} \times \bcp^n)$ gives an a.c.s. on $S^{4p} \times \bcp^n$ with $a = a_1 + a_2+a_3$. As we argued in the proof of Theorem \ref{4m}, we have $a_2 =0$. By  Remark \ref{remark1} we have that if $n>1$ then $4\cdot(2p-1)!$ divides each Chern class $c_i(a_1)$. Hence in the case $n>1$, we have that $4\cdot(2p-1)!$ divides the Euler characteristic $\chi(S^{4p} \times \bcp^n) = 2(n+1)$. In the case when $n =1$, we know that $S^{4p} \times \bcp^1$ has a.c.s only when $p =1$. This completes the proof of the proposition.	
\end{proof}

Next we recall that for any element $a_3 \in \wt{K}(\bcp^n)$ such that $\rho(a_3) = [\tau_{\bcp^n}]$, the total Chern class of $a_3$ has been described on p.130 of \cite{thomas} as  
 \begin{equation} \label{chern of a_3}
 c(a_3) = (1-x)^{n+1}(1\pm (n-1)!x^{n})^{ud_{r+1}}\prod_{1\leq k \leq r}\Big(\frac{1+kx}{1-kx} \Big)^{d_k},
 \end{equation}
 where $d_i$'s are integers and $$\begin{array}{cccl}
 u & = & 0 & \mbox{ if $n$ is even}, \\
   & = & 1 & \mbox{ if $n \equiv 3 \pmod 4$}, \\
   & = & 2 & \mbox{ if $n \equiv 1 \pmod 4$}.
 \end{array}$$
 We remark that there is a typographical error in the signs in the first two products of the right hand side of  Equation \ref{chern of a_3} as expressed in \cite{thomas}. For example one can easily compute $c(\eta^n) = (1 + (n-1)!x^{n})$ when $n = 1,3$ whereas $c(\eta^2) = (1 -x^{2})$ when $n = 2$.
 
 % is written as $(1+x)$ in \cite{thomas} which is a . Further we also remark that the sign in the middle term  $(1\pm (n-1)!x^{n})^{ud_{r+1}}$ of the right hand side of the Equation \ref{chern of a_3} will depend upon $n$. For example if  
 We know that $S^{2m} \times \bcp^1$ has a.c.s. if and only if $m = 1,2$ and $3$.  Next we give the proof of Theorem \ref{finalthm}.

\begin{proof}[{\bf Proof of Theorem \ref{finalthm}}] First note that $S^2 \times \bcp^n$, $S^6 \times \bcp^n$ have a.c.s.  By Corollary  \ref{cor2} and Proposition \ref{2nd prop}, only  thing that remains to complete the proof of the theorem is to show that  for $q>0$ there is no a.c.s on $S^4 \times \bcp^{4q+1}$. Suppose  $a \in \wt{K}(S^4 \times \bcp^{4q+1})$ gives an a.c.s. on $S^4 \times \bcp^{4q+1}$ with $a = a_1 +a_2 +a_3$. As observed in the proof of  Theorem \ref{4m}, we have $a_2 = 0$. The total Chern class of $a_1$ is given in (3) of Proposition \ref{chernclasses} as
$$c(a_1) =1 - 2 \sum_{i= 1}^{2q+1}2i(\sum_{k=1}^{2q} b_k k^{2i-1})y_2x^{2i-1}  \pm (4q+2)! \cdot b_{2q+1} y_2x^{4q+1}$$ where $b_i$'s are integers.  The total Chern class of $a_3$ as  described in Equation \ref{chern of a_3} is as follows
 $$c(a_3) = (1-x)^{4q+2}(1\pm (4q)!x^{4q+1})^{2d_{2q+1}}\prod_{1\leq k \leq 2q}\Big(\frac{1+kx}{1-kx} \Big)^{d_k},$$ 
 where $d_i$'s are integers. After simplifying we can write
 $$c(a_3) = 1+ \sum_{i =1}^{2q} \binom{4q+2}{2i}x^{2i} + A$$ where $A$ is a polynomial in $x$ with even coefficients. As $c(a) = c(a_1)c(a_3)$, the top Chern class
 \begin{equation}\label{sum}
  c_{4q+3}(a) =  ( -2 \sum_{i=1}^{2q+1}2i\binom{4q+2}{2i} \sum_{k =1}^{2q}b_k k^{2i-1}+ h)y_2 x^{4q+1}
 \end{equation}
  where $h$ is a multiple of $8$. We shall next prove that the coefficient $$h_k : = -2\sum_{i =1}^{2q+1}2i\binom{4q +2}{2i} k^{2i-1},$$ of $b_ky_2x^{4q+1}$ in  Equation \ref{sum} is a multiple of $8$. Clearly, when $k$ is even then $h_k$ is a multiple of $8$. Next consider the case when $k$ is odd. Consider the following equality.
 $$ (4q+2)(1+k)^{4q+1} = \sum_{l =1}^{4q+2}\binom{4q+2}{l} l k^{l-1}.$$ As $k$ is odd, the left hand summation is a multiple of 4. The right hand summation can be decomposed into three parts as follows: 
 $$\begin{array}{cc}  \sum_{i =1}^{2q+1} 2i \binom{4q+2}{2i}  k^{2i-1} & + \hspace{.4cm} \sum_{j =1}^{q}\binom{4q+2}{2j -1} \Big[(2j-1) k^{2j-2} + (4q-2j+3) k^{4q -2j +2} \Big ] \\ & \\ & + \hspace{.4cm} \binom{4q+2}{2q+1} (2q+1) k^{2q}.
 \end{array}$$
 The middle term in the above summation is a multiple of $4$. To see that $\binom{4q+2}{ 2q+1}$ is a multiple of $4$, we write $$\binom{4q+2}{2q+1} = \frac{(4q+2)(4q+1)}{(2q+1)^2}\binom{4q}{2q}$$ and use the fact that $\binom{4q}{2q}$ is even which can be seen from the following equality 
 $$ 2^{4q} = (1 +1)^{4q} = \sum_{i =0}^{4q} \binom{4q}{i}.$$ This completes the proof that each $h_k$ is a multiple of $8$. Since $a$ gives an a.c.s. on $S^4 \times \bcp^{4q+1}$, the top Chern class $c_{4q+3}(a)$ is multiple of $8$. This implies that $8$ divides the Euler characteristic $2(4q+2)$. But this is a contradiction. Hence there is no a.c.s. on $S^4 \times \bcp^{4q+1}$. This completes the proof of the theorem.
\end{proof}

The above construction helps us to give an explicit way to obtain  almost complex structures on $S^{2m} \times \bcp^n$ for few values of $m$ and $n$. In particular when $m =2$ and $n=3$ we show that $S^4 \times \bcp^3$ has infinite number of almost complex structures. We fix an orientation on $S^{2m}$ and $\bcp^n$ such that $e(S^{2m}) = -2y_m$ and $e(\bcp^n) = (-1)^n(n+1)x^n$. We  fix an orientation on $S^{2m} \times \bcp^n$ arising by taking the orientation on each factor.

%Let $\xi$ be a $m+n$ complex vector bundle with  $[\xi]=a_1+a_2+a_3,$ where $a_1\in\wt{K}(S^{2m}\wedge\bcp^n),$ $a_2\in\wt{K}(S^{2m})$ and $a_3\in\wt{K}(\bcp^n).$ The vector bundle $\xi$ will give an a.c.s. if and only if $\rho(a_1)=\rho(a_2)=0$ and $\rho(a_3)=[\tau_{\bcp^n}]$ and the top Chern class $c_{m+n}(\xi)=2(n+1).$ 

\begin{Exm} \label{example1} {\em
 Let $m =1$ and $n = 1$.  From Proposition \ref{prop1} and \ref{prop2}, we get that $a = a_1 + a_2 + a_3 \in \wt{K}(S^2 \times \bcp^1)$ gives an a.c.s. on $S^2 \times \bcp^1 = S^2 \times S^2$ if and only if there are integers $d_1$, $d_2$ such that $a_1=0$, $a_2=2d_1\eta$ and $a_3=[T_{\bcp^1}]+2d_2\eta$ and $c_2(a) = e(S^2 \times S^2) = (-2y_1)(-2x) = 4y_1x.$ This gives the  equation 
$$4d_1(d_2-1)=4,$$
which has the following two solutions: $d_1=1$, $d_2=2$ and $d_1=-1$, $d_2=0$. 
Therefore we have the following two possibilities: $a=2\eta+[T_{\bcp^1}] + 4 \eta$ and $a=-2\eta+[T_{\bcp^1}]$. By stability property, we thus have that there are exactly two non-isomorphic a.c.s. on $S^2 \times \bcp^1$. Hence the only  two almost complex structures are   $\bar{T}_{S^2} \oplus \bar{T}_{\bcp^1}$ and  $T_{S^2} \oplus T_{\bcp^1}$. Here $\bar{T}_{S^2}$ is complex conjugate bundle of $T_{S^2}$ and $[\bar{T}_{\bcp^1}] = [T_{\bcp^1}] + 4 \eta$.
We note here that if we change the orientation on $S^2 \times S^2$, a similar argument will say that again it has exactly two almost complex structures which will be given by $T_{S^2} \oplus \bar{T}_{\bcp^1}$ and 
$\bar{T}_{S^2} \oplus T_{\bcp^1}$. The existence of exactly two almost complex structures on $S^2 \times S^2$  was also observed by Sutherland in  \cite{sutherland}.}
\end{Exm}

\begin{Exm}{\em 
Let $m=1$ and $n=2$. As in  Example \ref{example1}, an element $a \in \wt{K}(S^2 \times \bcp^2)$  gives an a.c.s. on $S^2 \times \bcp^2$ if and only if there are integers $b_1$, $d_1$ and $d_2$ such that $a_1=b_1(g(H-1)-\overline{g(H-1)})$, $a_2=2d_1\eta$, $a_3=[T_{\bcp^2}]+d_2(H-\bar{H})$ and $c_3(a)=-6y_1x^2.$ This gives the following equations
$$\begin{array}{cc}
b_1+d_1(-4d_2+4\binom{d_2}{2}+3)=-3 & \mbox{ when }d_2\geq0 \\\\

b_1+d_1(-8d_2+4\binom{-d_2}{2}+3)=-3 & \mbox{ when }d_2<0.
\end{array}$$
Clearly, we shall get infinite number of solutions of the above equations and by stability property, we have infinite number of almost complex structures on $S^2 \times \bcp^2$. The solution $b_1 = 0$, $d_1 = -1$ and $d_2= 0$ corresponds to the a.c.s. given by the holomorphic tangent bundle $T_{S^2} \oplus T_{\bcp^2}$. Similarly, if  we change the orientation on $S^2 \times \bcp^2$, we shall again get infinite number of almost complex structures.}
%If we choose $b_1 = 0$ then there are 4 solutions of the above equation given by $(d_1, d_2)$ belonging to the set $\{( -1,0), (-1,3), (3,2), (3,1)\}$. By Thomas \cite{thomas}, there are two distinct almost complex structures $\xi_1$ and $\xi_2$ on $\bcp^2$.   These two almost complex structures on $\bcp^2$ together with the unique a.c.s. on $S^2$ will give two different a.c.s. on $S^2 \times \bcp^2$ given by $T_{S^2} \oplus \xi_1$, $T_{S^2} \oplus \xi_2$,  and these almost complex structures are in one to one correspondence to a.c.s. obtained by taking $b_1 = 0$ in the above equation.
\end{Exm}

\begin{Exm}\label{example3} {\em
Let $m=2$ and $n=3$. Going along the same line of arguments as we did in last two examples we observe that number  of almost complex structures on $S^4\times\bcp^3$ is in one to one  correspondence with the number of solutions of the following equation
$$\begin{array}{cc}
b_1(4-3d_1+2\binom{d_1}{2})+3!b_2=-1& \mbox{ when }d_1\geq0\\\\
b_1(4-5d_1+2\binom{-d_1}{2})+3!b_2=-1& \mbox{ when }d_1<0.
\end{array}
$$
 For each $k \in \mathbb Z$, by taking $b_1=-7 +6k$, $b_2=1-k$ and $d_1=1$  we get infinite number of solutions of the above equation and thus this gives  infinite number of almost complex structures on $S^4\times\bcp^3.$  Again the reverse orientation on $S^4 \times \bcp^3$ has infinite number of almost complex structures. We remark that the existence of an a.c.s. on $S^4 \times \bcp^3$  also follows from Theorem 2 of \cite{heaps} and this was observed by Tang in \cite{tang}.}
\end{Exm}

We end this section with a result on nonexistence of a.c.s. on the Dold manifold $D(m,n)$.  It is known  (see \cite{dold}) that $D(m,n)$ is orientable if and only if $m+n$ is odd.  So an even dimensional orientable Dold manifold is of the form $D(2p, 2q+1)$ for some $p,q \geq 0$.  By considering  the 2-fold covering map $S^{2p} \times \bcp^{2q+1} \rightarrow D(2p,2q+1)$, we note that the  nonexistence of an a.c.s. on $S^{2p} \times \bcp^{2q+1}$ will imply the nonexistence of an a.c.s. on $D(2p,2q+1)$.  We have the following result.
 \begin{Thm} Let $p >0$. Then
$D(2p,2q+1)$ has no a.c.s. if one of the following is true
\begin{enumerate}
\item $p$ is odd.
\item $p \equiv 0 \pmod 4$ and $(q+1)$ is not a multiple of $2^{r-2} \cdot (p-1)!$, where $2^r$ is the maximum power of $2$ dividing $p$. 
\item $p \equiv 2 \pmod 4$ and $(q+1)$ is not a multiple of $(p-1)!$.
\item $p = 2$ and $q$ is even.
\end{enumerate}
\end{Thm}
 \begin{proof}
Statement (1) follows from Theorem 3.2 of \cite{tangdold}. The proof of  statements (2), (3) and (4) follow from Theorem \ref{4m}, Proposition \ref{2nd prop} and Theorem \ref{finalthm} respectively.  \end{proof}
 
 From the above theorem  it is clear that $D(2p,2q+1)$ does not admit an a.c.s. if $p$ is odd or $q$ is even.

 \subsection*{Acknowledgement:} We are grateful to Aniruddha Naolekar for several valuable discussions and useful comments.

\end{document}